\documentclass[11pt]{article}

\usepackage{enumerate}
\usepackage{amsmath}
\usepackage{amssymb,latexsym}
\usepackage{amsthm}
\usepackage{color}

\usepackage{graphicx}

\usepackage{hyperref}

\usepackage{amscd}

\DeclareMathOperator{\N}{N}

\title{New maximum scattered linear sets of the projective line}
\author{Bence Csajb\'ok, Giuseppe Marino and Ferdinando Zullo\thanks{\textcolor{black}{The
research  was supported by
Ministry for Education, University and Research of Italy MIUR (Project
PRIN 2012 "Geometrie di Galois e strutture di incidenza") and by the Italian National
Group for Algebraic and Geometric Structures and their Applications (GNSAGA
- INdAM).}}}
\date{}

\newcommand{\cC}{{\mathcal C}}

\newcommand{\cN}{{\mathcal N}}

\newcommand{\cF}{{\mathcal F}}

\newcommand{\F}{{\mathbb F}}

\newcommand{\la}{\langle}
\newcommand{\ra}{\rangle}

\newcommand{\ZG}{\mathcal{Z}(\mathrm{\Gamma L})}
\newcommand{\G}{\mathrm{\Gamma L}}

\newtheorem{theorem}{Theorem}[section]

\newtheorem{lemma}[theorem]{Lemma}
\newtheorem{corollary}[theorem]{Corollary}
\newtheorem{definition}[theorem]{Definition}
\newtheorem{proposition}[theorem]{Proposition}
\newtheorem{result}[theorem]{Result}

\DeclareMathOperator{\PG}{{PG}}

\DeclareMathOperator{\GL}{{GL}}

\DeclareMathOperator{\PGaL}{P\Gamma L}

\DeclareMathOperator{\GaL}{\Gamma L}

\begin{document}
\maketitle

\begin{abstract}

In \cite{BL2000} and \cite{LP2001} are presented the first two families of maximum scattered $\F_q$-linear sets of the projective line $\PG(1,q^n)$. More recently in \cite{Sh} and in \cite{CMPZ}, new examples of maximum scattered $\mathbb{F}_q$-subspaces of $V(2,q^n)$ have been constructed, but the equivalence problem of the corresponding linear sets is left open.

Here we show that the $\F_q$-linear sets presented in \cite{Sh} and in \cite{CMPZ}, for $n=6,8$, are new.
Also, for $q$ odd, $q\equiv \pm 1,\,0 \pmod 5$, we present new examples of maximum scattered $\F_q$-linear sets in $\PG(1,q^6)$,  arising from trinomial polynomials,  which define new $\F_q$-linear MRD-codes of $\F_q^{6\times 6}$ with dimension $12$, minimum distance 5 and middle nucleus (or left idealiser) isomorphic to $\F_{q^6}$.
\end{abstract}



\section{Introduction}
\label{sec:Intro}

Linear sets are natural generalisations of subgeometries.
Let $\Lambda=\PG(W,\F_{q^n})\allowbreak=\PG(r-1,q^n)$, where $W$ is a vector space of dimension $r$ over $\F_{q^n}$.
A point set $L$ of $\Lambda$ is said to be an \emph{$\F_q$-linear set} of $\Lambda$ of rank $k$ if it is
defined by the non-zero vectors of a $k$-dimensional $\F_q$-vector subspace $U$ of $W$, i.e.
\[L=L_U=\{\la {\bf u} \ra_{\mathbb{F}_{q^n}} \colon {\bf u}\in U\setminus \{{\bf 0} \}\}.\]
The maximum field of linearity of an $\F_q$-linear set $L_U$ is $\F_{q^t}$ if $t \mid n$ is the largest integer such that $L_U$ is an $\F_{q^t}$-linear set.
Two linear sets $L_U$ and $L_W$ of $\PG(r-1,q^n)$ are said to be \emph{$\mathrm{P\Gamma L}$-equivalent} (or simply \emph{equivalent}) if there is an element $\phi$ in $\mathrm{P\Gamma L}(r,q^n)$ such that $L_U^{\phi} = L_W$. It may happen that two $\F_q$--linear sets $L_U$ and $L_W$ of $\PG(r-1,q^n)$ are equivalent even if the two $\F_q$-vector subspaces $U$ and $W$ are not in the same orbit of $\Gamma \mathrm{L}(r,q^n)$ (see \cite{CSZ2015} and \cite{CMP} for further details).
In the recent years, starting from the paper \cite{Lu1999} by Lunardon, linear sets have been used to construct or characterise various objects in finite geometry, such as blocking sets and multiple blocking sets in finite projective spaces, two-intersection sets in finite projective spaces, translation spreads of the Cayley Generalized Hexagon, translation ovoids of polar spaces, semifield flocks and finite semifields. For a survey on linear sets we refer the reader to \cite{OP2010}, see also \cite{Lavrauw}. It is clear that in the applications it is crucial to have methods to decide whether two linear sets are equivalent or not.

In this paper we focus on {\it maximum scattered} $\F_q$-linear sets of $\PG(1,q^n)$ with maximum field of linearity $\F_q$, that is,
$\F_q$-linear sets of rank $n$ of $\PG(1,q^n)$ of size $(q^n-1)/(q-1)$. If $L_U$ is a maximum scattered $\F_q$-linear set, then $U$ is a {\it maximum scattered} $\F_q$-subspace.

If $\la (0,1) \ra_{\F_{q^n}}$ is not contained in the linear set $L_U$ of rank $n$ of $\PG(1,q^n)$ (which we can always assume after a suitable projectivity), then $U=U_f:=\{(x,f(x))\colon x\in \F_{q^n}\}$ for some $q$-polynomial $f(x)=\sum_{i=0}^{n-1}a_ix^{q^i}\in \F_{q^n}[x]$. In this case we will denote the associated linear set by $L_f$. The known non-equivalent (under $\Gamma\mathrm{L}(2,q^n)$) maximum scattered $\F_q$-subspaces are
\begin{enumerate}
\item $U^{1,n}_s:= \{(x,x^{q^s}) \colon x\in \F_{q^n}\}$, $1\leq s\leq n-1$, $\gcd(s,n)=1$ (\cite{BL2000,CSZ2016}),
\item $U^{2,n}_{s,\delta}:= \{(x,\delta x^{q^s} + x^{q^{n-s}})\colon x\in \F_{q^n}\}$, $n\geq 4$, $\N_{q^n/q}(\delta)\notin \{0,1\}$ \footnote{This condition  implies $q\neq 2$.}, $\gcd(s,n)=1$ (\cite{LP2001} for $s=1$, \cite{Sh,LTZ} for $s\neq 1$),
\item $U^{3,n}_{s,\delta}:= \{(x,\delta x^{q^s}+x^{q^{s+n/2}})\colon x\in \F_{q^{n}}\}$, $n\in \{6,8\}$, $\gcd(s,n/2)=1$, $\N_{q^n/q^{n/2}}(\delta) \notin \{0,1\}$, for the precise conditions on $\delta$ and $q$ see \cite[Theorems 7.1 and 7.2]{CMPZ} \footnote{Also here $q>2$, otherwise $L^{3,n}_{s,\delta}$ is not scattered.}.
\end{enumerate}

The stabilisers of the $\F_q$-subspaces above in the group $\GL(2,q^n)$ were determined in \cite[Sections 5 and 6]{CMPZ}. They have the following orders:
\begin{enumerate}
  \item for $U^{1,n}_s$ we have a group of order $q^n-1$,
  \item for $U^{2,n}_{s,\delta}$ we have a group of order $q^2-1$,
  \item for $U^{3,n}_{s,\delta}$ we have a group of order $q^{n/2}-1$.
\end{enumerate}

It is known, that for $n=3$ the maximum scattered $\F_q$-spaces of $V(2,q^3)$ are $\G(2,q^3)$-equivalent to $U^{1,3}_1$ (cf. \cite{LaVa2010}), and
for $n=4$ they are $\GL(2,q^4)$-equivalent either to $U^{1,4}_1$ or to $U^{2,4}_{1,\delta}$ (cf. \cite{CSZ2017}).

To make notation easier, by  $L^{i,n}_{s}$ and $L^{i,n}_{s,\delta}$ we will denote the $\F_q$-linear set defined by $U^{i,n}_{s}$ and $U^{i,n}_{s,\delta}$, respectively. The $\F_q$-linear sets
equivalent to $L^{1,n}_s$ are called \emph{of pseudoregulus type}. It is easy to see that $L^{1,n}_1=L^{1,n}_s$ for any $s$ with $\gcd(s,n)=1$ and that $U^{2,n}_{s,\delta}$ is $\GL(2,q^n)$-equivalent to $U^{2,n}_{n-s,\delta^{-1}}$.

In \cite[Theorem 3]{LP2001} Lunardon and Polverino proved that $L^{2,n}_{1,\delta}$ and $L^{1,n}_1$ are not $\mathrm{P}\Gamma \mathrm{L}(2,q^n)$-equivalent when $q>3$, $n\geq 4$. For $n=5$, in \cite{CMPxxxx} it is proved that  $L^{2,5}_{2,\delta}$ is $\mathrm{P}\Gamma\mathrm{L}(2,q^5)$-equivalent neither to $L^{2,5}_{1,\delta'}$ nor to $L^{1,5}_1$.

\medskip
In the first part of this paper we prove that for $n=6,8$ the linear sets $L^{1,n}_1$, $L^{2,n}_{s,\delta}$ and $L^{3,n}_{s',\delta'}$ are pairwise non-equivalent for any choice of $s,s',\delta, \delta'$.

In the second part of this paper we prove that the $\F_q$-linear set defined by
\[U^4_b:=\{(x, x^q +x^{q^3}+b x^{q^5})\colon x\in \F_{q^6}\}\]
with $b^2+b=1$, $q\equiv 0,\pm 1 \pmod 5$ is maximum scattered in $\PG(1,q^6)$ and it is not $\mathrm{P}\Gamma\mathrm{L}(2,q^6)$-equivalent to any previously known example. Connections between scattered $\F_q$-subspaces and MRD-codes have been investigated in \cite{Sh, CSMPZ2016, Lu2017}.
Using the relation found in \cite{Sh} we also present new examples of such codes.

\section{\texorpdfstring{Classes of $\F_q$-linear sets of rank $n$ of $\PG(1,q^n)$ and preliminary results}{Classes of GF(q)-linear sets of rank n of PG(1,qn) and preliminary results}}
\label{subdual}

For $\alpha\in \F_{q^n}$ and a divisor $h$ of $n$ we will denote by $\N_{q^n/q^h}(\alpha)$ the norm of $\alpha$ over the subfield $\F_{q^h}$, that is,
$\N_{q^n/q^h}(\alpha)=\alpha^{1+q^h+\ldots+q^{n-h}}$.

By \cite{BGMP2015,CMP} for $f(x)=\sum_{i=0}^{n-1}a_i x^{q^i}$ and $\hat{f}(x)=\sum_{i=0}^{n-1}a_i^{q^{n-i}}x^{q^{n-i}}$, the $\F_q$-subspaces
$U_f=\{(x,f(x))\colon x\in \F_{q^n}\}$ and $U_{\hat{f}}=\{(x,\hat{f}(x))\colon x\in \F_{q^n}\}$ define the same linear set of $\PG(1,q^n)$.
On the other hand $U_f$ and $U_{\hat{f}}$ are not necessarily $\Gamma\mathrm{L}(2,q^n)$-equivalent (see \cite[Section 3.2]{CMP}) and this motivates the following definitions.

\begin{definition}(\cite{CMP})\label{GL-class}
Let $L_U$ be an $\mathbb{F}_q-$linear set of $\PG(W,\mathbb{F}_{q^n})=\PG(1,q^n)$ of rank $n$ with maximum field of linearity $\mathbb{F}_q$.

We say that $L_U$ is of $\G$-\emph{class} $s$ if $s$ is the greatest integer such that there exist $\mathbb{F}_q-$subspaces $U_1,\ldots,U_s$ of $W$ with $L_{U_i}=L_U$ for $i \in \{1,\ldots,s\}$ and there is no $f \in \Gamma \mathrm{L}(2,q^n)$ such that $U_i=U_j^f$ for each $i\neq j$, $i,j \in \{1,2,\ldots,s\}$. If $L_U$ has $\Gamma \mathrm{L}$-class one, then $L_U$ is said to be \emph{simple}.

We say that $L_U$ is of $\mathcal{Z}(\GaL)$-class $r$ if $r$ is the greatest integer such that there exist $\F_q$-subspaces $U_1, U_2, \ldots, U_r$ of $W$ with $L_{U_i}=L_U$ for $i \in \{1,2,\ldots,r\}$ and $U_i \neq \lambda U_j$ for each $\lambda \in \F_{q^n}^*$ and for each $i \neq j$, $i, j \in \{1,2,\ldots,r\}$.

\end{definition}

\begin{result}(\cite[Prop. 2.6]{CMP})
Let $L_U$ be an $\F_q$-linear set of $\PG(1,q^n)$ of rank $n$ with maximum field of linearity $\F_q$ and let
$\varphi$ be a collineation of $\PG(1,q^n)$. Then $L_U$ and $L_U^{\varphi}$ have the same $\ZG$-class and $\G$-class. Also,
the $\G$-class of an $\F_q$-linear set cannot be greater than its $\ZG$-class.
\end{result}

For a $q$-polynomial $f(x)=\sum_{i=0}^{n-1}a_i x^{q^i}$ over $\F_{q^n}$ let $D_f$ denote the associated \emph{Dickson matrix} (or \emph{$q$-circulant matrix})
\[D_f:=
\begin{pmatrix}
a_0 & a_1 & \ldots & a_{n-1} \\
a_{n-1}^q & a_0^q & \ldots & a_{n-2}^q \\
\vdots & \vdots & \vdots & \vdots \\
a_1^{q^{n-1}} & a_2^{q^{n-1}} & \ldots & a_0^{q^{n-1}}
\end{pmatrix}
.\]
The rank of the matrix $D_f$ equals the rank of the $\F_q$-linear map $f$, see for example \cite{wl}.

We will use the following results.

\begin{proposition}
\label{Di}
Let $f$ and $g$ be two $q$-polynomials over $\mathbb{F}_{q^n}$. Then $L_f \subseteq L_g$ if and only if
\[x^{q^n}-x \mid \det D_{F(Y)}(x) \in \F_{q^n}[x],\]
where $F(Y)=f(x)Y-g(Y)x$.
In particular, if $\deg \det D_{F(Y)}(x) < q^n$, then $L_f \subseteq L_g$ if and only if $\det D_{F(Y)}(x)$ is the zero polynomial.
\end{proposition}
\begin{proof}
$L_f \subseteq L_g$ if and only if
 \[\left\{ \frac{f(x)}{x} \colon x \in \mathbb{F}_{q^n}^* \right\} \subseteq \left\{ \frac{g(x)}{x} \colon x \in \mathbb{F}_{q^n}^* \right\},\]
which means that $\displaystyle \frac{g(y)}{y}=\frac{f(x)}{x}$ can be solved in $y$ if we fix $x \in \mathbb{F}_{q^n}^*$.
Fix $x \in \mathbb{F}_{q^n}^*$, then the $q$-polynomial $F(Y)=f(x)Y-g(Y)x$ has rank less than $n$
since it has a non-zero solution.
Since the Dickson matrix $D_{F(Y)}(x)$ of $F(Y)$ has the same rank as $F(Y)$, it follows that $\det D_{F(Y)}(x)=0$ for each $x$.
It follows that $x^{q^n}-x \mid \det D_{F(Y)}(x)$.
\end{proof}

\begin{lemma}{\cite[Lemma 3.6]{CMP}}
\label{LemmaCMP}
Let $\displaystyle f(x)=\sum_{i=0}^{n-1} a_i x^{q^i}$ and $\displaystyle g(x)=\sum_{i=0}^{n-1} b_i x^{q^i}$ be two $q$-polynomials over $\mathbb{F}_{q^n}$ such that $L_f=L_g$. Then
\begin{equation}
\label{6}
a_0=b_0,
\end{equation}
for $k=1,2,\ldots,n-1$ it holds that
\begin{equation}
\label{7}
a_ka_{n-k}^{q^k}=b_kb_{n-k}^{q^k},
\end{equation}
for $k=2,3,\ldots,n-1$ it holds that
\begin{equation}
\label{8}
a_1a_{k-1}^qa_{n-k}^{q^k}+a_ka_{n-1}^qa_{n-k+1}^{q^k}=b_1b_{k-1}^qb_{n-k}^{q^k}+b_kb_{n-1}^qb_{n-k+1}^{q^k}.
\end{equation}
\end{lemma}

\section{\texorpdfstring{The $L^{2,n}_{s,\delta}$-linear sets in $\PG(1,q^n)$, $n=6,8$}{The L2nsdelta-linear sets in PG(1,qn)}}
\label{LPS}

In this section we determine the $\ZG$-class of the maximum scattered $\F_q$-linear sets of $\PG(1,q^n)$, $n=6,8$, introduced by Lunardon and Polverino, and generalised by Sheekey. Recall that $U^{2,n}_{s,\delta}$ is $\GL(2,q^n)$-equivalent to $U^{2,n}_{n-s,\delta^{-1}}$, thus
it is enough to study the linear sets $L^{2,n}_{s,\delta}$ with $s<n/2$ and $\gcd(s,n)=1$.

\begin{proposition}
\label{classLP6}
If $n=6$, then the $\ZG$-class of $L^{2,6}_{1,\delta}$ is two.
\end{proposition}
\begin{proof}
Since $g(x)=\delta x^q+x^{q^5}$ and $\hat{g}(x)=\delta^{q^5}x^{q^5}+x^q$ define the same linear set, we know
$L^{2,6}_{1,\delta}=L^{2,6}_{5,\delta^{q^5}}$. Suppose $L_f = L^{2,6}_{1,\delta}$ for some $f(x)=\sum_{i=0}^5 a_ix^{q^i} \in \F_{q^6}[x]$.
We show that there exists $\lambda\in \F_{q^6}^*$ such that either $\lambda U_f=U^{2,6}_{1,\delta}$ or $\lambda U_f=U^{2,6}_{5,\delta^{q^5}}$.

By \eqref{6} we obtain $a_0=0$, by \eqref{7} with $k=1,3$ we have
\begin{equation}
\label{eq111}
a_1a_5^q=\delta
\end{equation}
and $a_3=0$, respectively.
Also, with $k=2$ in \eqref{7} and \eqref{8}, taking \eqref{eq111} into account, we get $a_2=a_4=0$.

By Proposition \eqref{Di} we get that the Dickson matrix associated to the $q$-polynomial
\[F(Y)=\left(\frac{\delta}{a_5^q} x^q + a_5x^{q^5}\right)Y-x\left(\delta Y^q+Y^{q^5}\right)\]
has zero determinant for each $x\in \F_{q^6}$. Direct computation shows that this determinant is
\[\N_{q^6/q}(x/a_5)\left( \N_{q^6/q}(a_5)-1 \right)\left(\N_{q^6/q}(a_5)-\N_{q^6/q}(\delta)\right),\]
which has degree less than $q^6$, thus it is the zero polynomial.
We have two possibilities:
\begin{enumerate}
  \item If $\N_{q^6/q}(a_5)=1$, then putting $a_5=\lambda^{q^5-1}$ we obtain $\lambda U_f=U^{2,6}_{1,\delta}$.
  \item If $\N_{q^6/q}(a_5/\delta)=1$, then choosing $a_5= \delta^{q^5} \lambda^{q^5-1}$ we get $\lambda U_f=U^{2,6}_{5,\delta^{q^5}}$.
\end{enumerate}
Because of the choice of $\delta$, that is $\N_{q^6/q}(\delta)\neq 1$, it follows that there is no $\mu\in \F_{q^6}$ such that
$\mu U^{2,6}_{1,\delta}=U^{2,6}_{5,\delta^{q^5}}$ and this proves that the $\ZG$-class of $L^{2,6}_{1,\delta}$ is exactly two.
\end{proof}

\begin{proposition}
\label{classLP81}
If $n=8$, then the $\ZG$-class of $L^{2,8}_{1,\delta}$ is two.
\end{proposition}
\begin{proof}
Since $g(x)=\delta x^q+x^{q^7}$ and $\hat{g}(x)=\delta^{q^7}x^{q^7}+x^q$ define the same linear set, we have
$L^{2,8}_{1,\delta}=L^{2,8}_{7,\delta^{q^7}}$. Suppose $L_f=L^{2,8}_{1,\delta}$ for some $f(x)=\sum_{i=0}^7 a_ix^{q^i} \in \F_{q^8}[x]$.
We show that there exists $\lambda\in \F_{q^8}^*$ such that either $\lambda U_f=U^{2,8}_{1,\delta}$ or $\lambda U_f=U^{2,8}_{7,\delta^{q^7}}$.

By \eqref{6} we obtain $a_0=0$, by \eqref{7} with $k=1$ we have
\begin{equation}
\label{eq222}
a_1a_7^q=\delta
\end{equation}
and with $k=4$ we get $a_4=0$.
Putting $k=2$ in \eqref{7} and \eqref{8}, taking \eqref{eq222} into account, we get $a_2=a_6=0$.
By \eqref{7} with $k=3$ we have $a_3a_5=0$.

If $a_3=0$, then $f(x)=a_1x^q+a_5x^{q^5}+a_7x^{q^7}$.
Using Proposition \ref{Di}, we get that the determinant of the Dickson matrix associated to the $q$-polynomial
\[F(Y)=(a_1x^q+a_5x^{q^5}+a_7x^{q^7})Y-x(a_1a_7^q Y^q+Y^{q^7})\]
is divisible by $x^{q^8}-x$. The coefficient of $x^{2(1+q+q^2+q^3)}$ after reducing the determinant modulo $x^{q^8}-x$ is 
$a_1^{1+q+q^2+q^7}a_5^{q^3+q^4+q^5+q^6}$, which is zero only when $a_5=0$ by \eqref{eq222}.

On the other hand, if $a_5=0$, then $L_f=L_{\hat f}$ gives $a_3=0$.


Then $\displaystyle f(x)= \frac{\delta}{a_7^q}x^q+a_7x^{q^7}$.
By Proposition \ref{Di}, arguing as in the previous proof,
\[ \N_{q^8/q}(x/a_7)\left( \N_{q^8/q}(a_7)-1\right) \left( \N_{q^8/q}(a_7)-\N_{q^8/q}(\delta)\right)\]
is the zero polynomial. We have two possibilities:
\begin{enumerate}
  \item If $\N_{q^8/q}(a_7)=1$, then putting $a_7=\lambda^{q^7-1}$, we obtain $\lambda U_f=U^{2,8}_{1,\delta}$.
  \item If $\N_{q^8/q}(a_7/\delta)=1$, then choosing $a_7=\delta^{q^7} \lambda^{q^7-1}$ we have $\lambda U_f=U^{2,8}_{7,\delta^{q^7}}$.
\end{enumerate}
Because of the choice of $\delta$, that is $\N_{q^8/q}(\delta)\neq 1$, it follows that there is no $\mu\in \F_{q^8}$ such that
$\mu U^{2,8}_{1,\delta}=U^{2,8}_{7,\delta^{q^7}}$ and this proves that the $\ZG$-class of $L^{2,8}_{1,\delta}$ is exactly two.
\end{proof}

\begin{proposition}
\label{classLP83}
If $n=8$, then the $\ZG$-class of $L^{2,8}_{3,\delta}$ is two.
\end{proposition}
\begin{proof}
Since $g(x)=\delta x^{q^3}+x^{q^5}$ and $\hat{g}(x)=\delta^{q^5}x^{q^5}+x^{q^3}$ define the same linear set, we know
$L^{2,8}_{3,\delta}=L^{2,8}_{5,\delta^{q^5}}$. Suppose $L_f=L^{2,8}_{3,\delta}$ for some $f(x)=\sum_{i=0}^7 a_ix^{q^i} \in \F_{q^8}[x]$.
We show that there exists $\lambda\in \F_{q^8}^*$ such that either $\lambda U_f=U^{2,8}_{3,\delta}$ or $\lambda U_f=U^{2,8}_{5,\delta^{q^5}}$.

By \eqref{6} we obtain $a_0=0$, by \eqref{7} with $k=3$ we have
\[a_3a_5^{q^3}=\delta
\]
and with $k=4$ we get $a_4=0$.
Putting $k=1$ and $k=2$ in \eqref{7} we get
\begin{equation}
\label{first}
a_1a_7=0 \mbox{ and } a_2a_6=0,
\end{equation}
respectively.
With $k=2$ and $k=3$ in \eqref{8} we obtain
\begin{equation}
\label{vi3}
a_1^{q+1}a_6^{q^2}+a_2a_7^{q+q^2}=0.
\end{equation}
and
\begin{equation}
\label{vii3}
a_1a_2^qa_5^{q^3}+a_3a_7^qa_6^{q^3}=0.
\end{equation}
By \eqref{vi3} and \eqref{vii3}, taking \eqref{first} into account, at most one of $\{a_1,a_2,a_6,a_7\}$ is non-zero.

Hence $f(x)=a_3x^{q^3}+a_5x^{q^5}+a_ix^{q^i}$ with $i\in\{1,2,6,7\}$.
For each $i\in\{1,2,6,7\}$, by Proposition \ref{Di}, the determinant of the Dickson matrix $D_{F(Y)}(x)$ with
$F(Y)=f(x)Y-x(a_3 a_5^{q^3} Y^{q^3}+Y^{q^5})$ is zero modulo $x^{q^8}-x$.
Then the following hold:
\begin{itemize}
  \item for $i=1$ the coefficient of $x^{3+3q+q^2+q^3}$ in the reduced form of $\det D_{F(Y)}(x)$ is $a_1^{1+q+q^2+q^7}a_3^{q^5+q^6}a_5^{q^3+q^4}$,
  \item for $i=2$ the coefficient of $x^{3+2q+q^2+q^3+q^4}$ in the reduced form of $\det D_{F(Y)}(x)$ is $a_2^{1+q+q^2+q^6+q^7}a_3^{q^5}a_5^{q^3+q^4}$.
\end{itemize}
Thus $a_i=0$ for $i\in \{1,2\}$ and since $L_f=L_{\hat f}$, the same holds for $i\in\{6,7\}$.  Then from \eqref{vi3} we get
$\displaystyle f(x)= \frac{\delta}{a_5^{q^3}}x^{q^3}+a_5x^{q^5}$.
By Proposition \ref{Di}, arguing as in the previous proof,
\[ \N_{q^8/q}(x/a_5)\left( \N_{q^8/q}(a_5)-1\right) \left( \N_{q^8/q}(a_5)-\N_{q^8/q}(\delta)\right)\]
is the zero polynomial. Then the following holds:
\begin{enumerate}
  \item If $\N_{q^8/q}(a_5)=1$, then putting $a_5=\lambda^{q^5-1}$ gives $\lambda U_f=U^{2,8}_{3,\delta}$.
  \item If $\N_{q^8/q}(a_5/\delta)=1$, then set $a_5=\delta^{q^5} \lambda^{q^5-1}$, and hence $\lambda U_f=U^{2,8}_{5,\delta^{q^5}}$.
\end{enumerate}
As in the previous proof, it can be easily seen that the $\ZG$-class is exactly two.
\end{proof}

\begin{theorem}
\label{perFerdi}
The linear set $L^{2,n}_{s,\delta}$ is not of pseudoregulus type for each $n,s,\delta, q$.
Also, the linear sets $L^{2,8}_{1,\delta}$ and $L^{2,8}_{3,\rho}$ are not $\PGaL(2,q^8)$-equivalent.
\end{theorem}
\begin{proof}
Suppose that $L^{2,n}_{s,\delta}$ is of pseudoregulus type. Then by \cite{LaShZa2013} there exists
an element $f$ of $\GL(2,q^n)$ such that $(U^{2,n}_{s,\delta})^f=U^{1,n}_r$ with $\gcd(r,n)=1$.
Since the $\F_{q^n}$-linear automorphism groups of $U^{2,n}_{s,\delta}$ and $(U^{2,n}_{s,\delta})^f$ are conjugated and since
the groups of $U^{1,n}_r$ and $U^{2,n}_{s,\delta}$ have orders $q^n-1$ and $q^2-1$, respectively (cf. Introduction),
we get a contradiction.

For the second part, suppose to the contrary that $L^{2,8}_{1,\delta}$ and $L^{2,8}_{3,\rho}$ are $\PGaL(2,q^8)$-equivalent.
Then by Proposition \ref{classLP83} there exists a field automorphism $\sigma$, an invertible matrix $\displaystyle \left( \begin{array}{llrr} a & b \\ c & d \end{array} \right)$ and $\alpha, \beta \in \F_{q^8}^*$ such that for each $x \in \F_{q^8}$ there exists $z \in \F_{q^8}$ satisfying
\[ \left(
\begin{array}{llrr}
a & b \\
c & d \end{array} \right)
\left( \begin{array}{c}
x^\sigma \\
\delta^\sigma x^{\sigma q} + x^{\sigma q^7}
\end{array} \right) =
\left( \begin{array}{c}
z \\
\alpha z^{q^3} + \beta z^{q^5}
\end{array} \right).\]
Equivalently, for each $x \in \F_{q^8}$
\[ cx^\sigma+d \delta^\sigma x^{\sigma q} + d x^{\sigma q^7}= \alpha (a^{q^3}x^{\sigma q^3}+\delta^{\sigma q^3}b^{q^3}x^{\sigma q^4}+b^{q^3}x^{\sigma q^2})+ \]
\[ \beta(a^{q^5} x^{\sigma q^5} + \delta^{\sigma q^5} b^{q^5} x^{\sigma q^6}+ b^{q^5} x^{\sigma q^4}). \]
This is a polynomial identity in $x^\sigma$. Comparing the coefficients of $x^{q^2}$ and $x^{q^3}$ we get that $a=b=0$, which is a contradiction.
\end{proof}

\section{\texorpdfstring{The $L^{3,n}_{s,\delta}$-linear sets in $\PG(1,q^n)$, $n=6,8$}{The L3nsdelta-linear sets in PG(1,qn)}}
\label{CMPZscatt}

In this section we determine the $\ZG$-class of the maximum scattered $\F_q$-linear sets of $\PG(1,q^n)$, $n=6,8$, introduced in \cite{CMPZ}.
According to \cite[Section 5, pg. 7]{CMPZ}, $U^{3,n}_{s,\delta}$ is $\GL(2,q^n)$-equivalent to $U^{3,n}_{n-s,\delta^{q^{n-s}}}$ and to
$U^{3,n}_{s+n/2,\delta^{-1}}$, thus it is enough to study the linear sets $L^{3,n}_{s,\delta}$ with $s<n/4$, $\gcd(s,n/2)=1$ and hence only with $s=1$ for $n=6,8$.

\begin{proposition}
\label{classCMPZ}
The $\ZG$-class of $L_g$, with $g(x)=\delta x^q+x^{q^4}$, $\delta\neq 0$, is two and hence the $\ZG$-class of $L^{3,6}_{1,\delta}$ is two as well.
Moreover, $L^{3,6}_{1,\delta}$ is a simple linear set.
\end{proposition}
\begin{proof}
Since $g(x)$ and $\hat{g}(x)=\delta^{q^5}x^{q^5}+x^{q^2}$ define the same linear set, we know
$L_g=L_{\hat g}$. Suppose $L_f = L_g$ for some $f(x)=\sum_{i=0}^5 a_ix^{q^i} \in \F_{q^6}[x]$.
We show that there exists $\lambda\in \F_{q^6}^*$ such that either $\lambda U_f=U_g$ or $\lambda U_f=U_{\hat g}$.

By \eqref{6}, we obtain $a_0=0$ and by \eqref{7} with $k=2$ we get $a_3=0$. Also, by \eqref{7} with $k=1$ and $k=2$, we have
\begin{equation}
\label{iiC}
a_1a_5=0
\end{equation}
and
\begin{equation}
\label{iiiC}
a_2a_4=0,
\end{equation}
respectively.
By \eqref{8} with $k=2$ we get
\begin{equation}
\label{vC}
a_1^{q+1}a_4^{q^2}+a_2a_5^{q+q^2}=\delta^{q+1}.
\end{equation}
From \eqref{iiC}, \eqref{iiiC} and \eqref{vC} it follows that either
\[\displaystyle f(x)=\frac{\delta^{q+1}}{a_5^{q+q^2}}x^{q^2}+a_5x^{q^5},\]
or
\[\displaystyle f(x)=a_1x^q+\left(\frac{\delta}{a_1}\right)^{q^5+q^4}x^{q^4}.\]
In both cases, the determinant of the Dickson matrix associated with $F(Y)=f(x)Y-x(\delta Y^q+Y^{q^4})$ is the zero-polynomial after reducing modulo $x^{q^6}-x$ and hence in the first case we obtain $\N_{q^6/q}(a_5/\delta)=1$, in the second case $\N_{q^6/q}(a_1/\delta)=1$.
In the former case $a_5=\delta^{q^5}\lambda^{q^5-1}$ and hence $\lambda U_f=U_{\hat g}$.
In the latter case $a_1=\delta \lambda^{q-1}$ implying $\lambda U_f=U_g$.

This means that the $\ZG$-class of $U_g$ is at most two. Straightforward computation shows that it is exactly two.
In case of $L^{3,6}_{1,\delta}$ (and hence with $\N_{q^6/q^3}(\delta)\neq 1$) it follows from \cite[Section 5]{CMPZ} that
$U^{3,6}_{1,\delta}$ and $U^{3,6}_{5,\delta^{q^5}}$ are $\G(2,q^6)$-equivalent and hence $L^{3,6}_{1,\delta}$ is simple.
\end{proof}

\begin{proposition}
\label{classCMPZa8}
The $\ZG$-class of $L_g$, with $g(x)=\delta x^q+x^{q^5}$, $\delta\neq 0$, is two and hence the $\ZG$-class of $L^{3,8}_{1,\delta}$ is two as well.
Moreover, $L^{3,8}_{1,\delta}$ is a simple linear set.
\end{proposition}
\begin{proof}
Since $g(x)=\delta x^q+x^{q^5}$ and $\hat{g}(x)=\delta^{q^7}x^{q^7}+x^{q^3}$ define the same linear set, we have $L_g=L_{\hat g}$. Suppose $L_f = L_g$ for some $f(x)=\sum_{i=0}^7 a_ix^{q^i} \in \F_{q^8}[x]$.
We show that there exists $\lambda\in \F_{q^8}^*$ such that either $\lambda U_f=U_g$ or $\lambda U_f=U_{\hat g}$.

By \eqref{6}, we obtain $a_0=0$ and by \eqref{7} with $k=4$ we get $a_4=0$.
Also, by \eqref{7} with $k=1$, $k=2$ and $k=3$ we get
\begin{equation}
\label{iib}
a_1a_7=a_2a_6=a_3a_5=0.
\end{equation}

By \eqref{8}, with $k=2$ we obtain
\begin{equation}
\label{vib}
a_1^{q+1}a_6^{q^2}+a_2a_7^{q+q^2}=0,
\end{equation}
and with $k=3$ we get
\begin{equation}
\label{viib}
a_1a_2^qa_5^{q^3}+a_3a_7^qa_6^{q^3}=0.
\end{equation}

By \eqref{iib}, first suppose $a_1=a_2=a_3=0$. Proposition \ref{Di} yields that the determinant of the Dickson matrix associated with
\[ F(Y)=(a_5x^{q^5}+a_6x^{q^6}+a_7x^{q^7})Y-x(\delta Y^q+Y^{q^5}),\]
has to be the zero polynomial after reducing modulo $x^8-x$.
The coefficient of $x^{1+2q+2q^2+2q^3+q^4}$ is $-a_5^{q^4+q^5+q^6+q^7}\delta^{1+q+q^2}$, hence $a_5=0$. The coefficient of $x^{1+q+2q^2+2q^3+q^4+q^5}$ is $-a_6^{q^4+q^5+q^6+q^7}\delta^{1+q+q^2}$, hence $a_6=0$.
The coefficient of $x^{1+q+q^2+2q^3+q^4+q^5+q^6}$ is $-a_7^{q^4+q^5+q^6+q^7}\delta^{1+q+q^2}$, hence $a_7=0$, a contradiction.

Now suppose $a_1=a_2=a_5=a_7=0$.
Again, Proposition \ref{Di} yields that the determinant of the Dickson matrix associated with
\[ F(Y)=(a_3x^{q^3}+a_6x^{q^6})Y-x(\delta Y^q+Y^{q^5}),\]
has to be the zero polynomial after reducing modulo $x^8-x$.
The coefficient of $x^{2+2q+3q^2+q^3}$ is $-a_3^{q^5+q^6+q^7}a_6^{q^4}\delta^{1+q+q^2}$, hence $a_3a_6=0$. We cannot have $a_3=0$ because of the previous paragraph, hence $a_6=0$, but then the coefficient of $x^{1+2q+2q^2+2q^3+q^4}$ is $-a_3^{1+q^5+q^6+q^7}\delta^{q+q^2+q^3}$. Then again $a_3=0$ follows, a contradiction.

Taking into account $L_f=L_{\hat f}$ and \eqref{iib}, \eqref{vib}, \eqref{viib}, two cases remain: $f(x)=a_3x^{q^3}+a_7x^{q^7}$ and $f(x)=a_1x^q+a_5x^{q^5}$.

In the former case Proposition \ref{Di} yields that the determinant of the Dickson matrix associated with
\[ F(Y)=(a_3x^{q^3}+a_7x^{q^7})Y-x(\delta Y^q+Y^{q^5}),\]
has to the zero polynomial after reducing modulo $x^8-x$.
The coefficient of $x^{2+2q+2q^2+2q^3}$ is $a_3^{q^5+q^6+q^7}a_7^{q^4}(a_3a_7^{q+q^2+q^3}-\delta^{1+q+q^2})$, hence
\[a_3=\delta^{1+q+q^2}/a_7^{q+q^2+q^3}.\]
Since the coefficient of $x^{2+q+2q^2+2q^3+q^5}$ is
\[(\N_{q^8/q}(\delta)-\N_{q^8/q}(a_7))
\delta^{2+q+q^2+q^6+2q^7}/a_7^{3+2q+2q^2+q^3+q^5+q^6+2q^7},\]
which has to be zero and hence it follows that $\N_{q^8/q}(a_7/\delta)=1$.
Then there exists $\lambda\in \F_{q^8}^*$ such that $a_7=\delta^{q^7}\lambda^{q^7-1}$ and hence $a_3=\lambda^{q^3-1}$, i.e. $\lambda U_f=U_{\hat g}$.

On the other hand, if $f(x)=a_1 x^q+a_5 x^{q^5}$, then the previous paragraph yields that there exists $\lambda \in \F_{q^8}^*$ such that
$\lambda U_{\hat f }=U_{\hat g}$ and hence $\lambda^{-1}U_f=U_g$.

Since there is no $\mu \in \F_{q^8}^*$ such that $U_g=\mu U_{\hat g}$, it follows that the $\ZG$-class of $U_g$ is exactly two.
In case of $L^{3,8}_{1,\delta}$ (and hence with $\N_{q^8/q^4}(\delta)\neq 1$) it follows from \cite[Section 5]{CMPZ} that
$U^{3,8}_{1,\delta}$ and $U^{3,8}_{7,\delta^{q^7}}$ are $\G(2,q^8)$-equivalent and hence $L^{3,8}_{1,\delta}$ is simple.
\end{proof}

\begin{theorem}
\label{conCZ}
The linear set $L^{3,n}_{1,\delta}$, $n=6,8$, is not of pseudoregulus type and not $\PGaL(2,q^n)$-equivalent to $L^{2,n}_{s,\rho}$.
\end{theorem}
\begin{proof}
Since the $\F_{q^n}$-linear automorphism group of $U^{3,n}_{1,\delta}$ has order $q^{n/2}-1$ (cf. \cite[Corollary 5.2]{CMPZ}), the same arguments as in
the proof of Theorem \ref{perFerdi} can be applied to show that $L^{3,n}_{1,\delta}$ is not of pseudoregulus type.

Suppose that $L^{3,n}_{1,\delta}$ is equivalent to $L^{2,n}_{s,\rho}$ for some $n\in \{6,8\}$. Then by Propositions \ref{classLP6}, \ref{classLP81}, \ref{classLP83}, there exists $f\in \G(2,q^n)$ such that either $(U^{3,n}_{1,\delta})^f=U^{2,n}_{s,\rho}$ or $(U^{3,n}_{1,\delta})^f=U^{2,n}_{n-s,\rho^{q^{n-s}}}$. This gives a contradiction, since the sizes of the corresponding automorphism groups are different.
\end{proof}

\section{\texorpdfstring{New maximum scattered linear sets in $\PG(1,q^6)$}{New maximum scattered linear sets in PG(1,q6)}}

In this section we show that $L_g$ with $g(x)=x^q+x^{q^3}+bx^{q^5} \in \F_{q^6}[x]$, $q$ odd, $q \equiv 0,\pm 1 \pmod 5$, $b^2+b=1$ is a maximum scattered $\F_q$-linear set of $\PG(1,q^6)$ which is not equivalent to any other previously known example. 
Note that, under these assumptions we have $b\in \F_q$.

The $\F_q$-subspace $U_g=\{(x,x^q+x^{q^3}+bx^{q^5}) \colon x \in \F_{q^6} \}$
is scattered if and only if for each $m\in \F_{q^6}$
\[\frac{x^q+x^{q^3}+bx^{q^5}}{x}=-m\]
has at most $q$ solutions. Those $m$ which admit exactly $q$ solutions correspond to points $\la(1,-m)\ra_{\F_{q^6}}$ of $L_g$ with weight one.
It follows that $U_g$ is scattered if and only if for each $m\in \F_{q^6}$ the kernel of
\[r_{m,b}(x):=mx+x^q+x^{q^3}+bx^{q^5}\]
has dimension less than two, or, equivalently, the Dickson matrix associated with
$r_{m,b}(x)$, that is,
\[ D_{m,b}=\left( \begin{array}{llllllrrrrrr}
m & 1 & 0 & 1 & 0 & b \\
b & m^q & 1 & 0 & 1 & 0 \\
0 & b & m^{q^2} & 1 & 0 & 1 \\
1 & 0 & b & m^{q^3} & 1 & 0 \\
0 & 1 & 0 & b & m^{q^4} & 1 \\
1 & 0 & 1 & 0 & b & m^{q^5} \\
\end{array}\right)\]
has rank at least $5$ for each $m \in \F_{q^6}$.

Denote by $M_{i,j}$ the determinant of the matrix obtained from $D_{m,b}$ by deleting the $i$-th row and $j$-th column and consider
the two minors:
\[
M_{6,3}=2-3b+(b-1)(\N_{q^6/q^3}(m)+\N_{q^6/q^3}(m^q))+\N_{q^6/q^3}(m)^{q+1}+(1-b)(m^{1+q}-m^{q^3+q^4}),
\]
and
\[
M_{6,4}=2m-3bm+2m^{q^2}-3bm^{q^2}+m^{q^4}-bm^{q^4}+m^{1+q+q^2}+bm^{1+q+q^4}+bm^{q+q^2+q^4}.\]
\medskip

\begin{theorem}
\label{main}
If $q \equiv 0,\pm 1 \pmod 5$, $q$ odd and $b^2+b=1$ (hence $b\in \F_q$), then $U_g$ is a maximum scattered $\F_q$-subspace for $g(x)=x^q+x^{q^3}+bx^{q^5}$.
\end{theorem}
\begin{proof}
It is sufficient to show that $M_{6,3}$ and $M_{6,4}$ cannot be both zeros for the same value of $m \in \F_{q^6}$.
If $m=0$, then $M_{6,3}=2-3b\neq 0$ since $b=2/3$ does not satisfy our condition.
First suppose that $M_{6,3}$ vanishes for some $m\in \F_{q^6}^*$. Then
\[ m^{1+q}-m^{q^3+q^4}=\frac{2-3b+(b-1)(\N_{q^6/q^3}(m)+\N_{q^6/q^3}(m^q))+\N_{q^6/q^3}(m)^{q+1}}{b-1},\]
and since the righ-hand side is in $\F_{q^3}$, the same follows for the left-hand side, and hence
$m^{1+q}-m^{q^3+q^4}=m^{q^3+q^4}-m^{1+q}$, from which $m^{1+q} \in \F_{q^3}^*$ follows.
So, if $m^{1+q} \notin \F_{q^3}^*$, then $rk(D_{m,b})\geq 5$.
Now, suppose $m^{1+q} \in \F_{q^3}^*$, then $M_{6,3}$ can be written as
\[ M_{6,3}= 2 - 3 b + (1-b)(-m^{1+q^3}-m^{q+q^4})+m^{2(1+q)}=((1-b)-m^{1+q^3})^{q+1}.\]
Since $M_{6,3}=0$, we have
\begin{equation}
\label{m}
1-b=m^{1+q^3}\in \F_q.
\end{equation}
Then $m^{(q^3+1)(q+1)}=m^{2(q+1)}=(1-b)^2$ and hence either $m^{q+1}=1-b$, or $m^{q+1}=b-1$. In both cases $m^{q+1} \in \F_q$ follows.
The latter case cannot hold. Indeed by \eqref{m} we would get $m^{q^3+1}=-m^{q+1}$, so $m^{q^2}=-m$, which holds only if $m\in \F_{q^4}\cap \F_{q^6}=\F_{q^2}$,  a contradiction. In the former case, by \eqref{m} we obtain $m^{q^3+1}=m^{q+1}$, so $m \in \F_{q^2}$.
It follows that, taking $m^{1+q}=1-b=b^2$ into account, $M_{6,4}=4m(1-b)$, which cannot be zero.
\end{proof}

\medskip

Similarly to the proof of \cite[Proposition 5.2]{CMPZ} it is easy to prove the following result.

\begin{proposition}
\label{grouptrin}
The linear automorphism group of $U_g$ (defined as in Theorem \ref{main}) is
\[\mathcal{G}=\left\{ \left( \begin{array}{llrr} \lambda & 0\\ 0 & \lambda^{q} \end{array} \right) \colon \lambda \in \F_{q^2}^* \right\}.\]
\qed
\end{proposition}

\begin{proposition}
\label{equiv}
The maximum scattered $\F_q$-subspace $U_g$ defined in Theorem \ref{main} is not $\GaL(2,q^6)$-equivalent to the $\F_q$-subspaces $U^{1,6}_{s}$, $U^{2,6}_{t,\rho}$ and $U^{3,6}_{h,\xi}$.
\end{proposition}
\begin{proof}
As in the proof of Theorem \ref{perFerdi}, the size of the linear automorphism group of $U_g$ is different from the size of the group
of $U^{1,6}_{s}$ and of $U^{3,6}_{h,\xi}$ (cf. Introduction), hence it remains to show that $U_g$ is not $\G(2,q^6)$-equivalent to $U^{2,6}_{t,\rho}$.

Since any $\F_q$-subspace of the form $U^{2,6}_{5,\eta}$ is $\GL(2,q^6)$-equivalent to $U^{2,6}_{1,\rho}$ for some $\rho$, it is enough to show that
$U_g$ and $U^{2,6}_{t,\rho}$ lie on different orbits of $\GaL(2,q^6)$.
Suppose the contrary, then there exist $\sigma \in \mathrm{Aut}(\F_{q^6})$ and an invertible matrix
$\left( \begin{array}{llrr} \alpha & \beta \\ \gamma & \delta \end{array} \right)$ such that for each $x \in \F_{q^6}$ there exists $z \in \F_{q^6}$ satisfying  \[ \left(
\begin{array}{llrr}
\alpha & \beta\\
\gamma & \delta
\end{array} \right)
\left(
\begin{array}{llr}
\hspace{1cm}x^\sigma\\
\rho^\sigma x^{\sigma q}+x^{\sigma q^5}
\end{array} \right)
=\left(
\begin{array}{llr}
\hspace{1cm}z\\
z^q+z^{q^3}+bz^{q^5} \end{array}
\right).\]
Equivalently, for each $x \in \F_{q^6}$ we have
\[ \gamma x^\sigma +\delta (\rho^{\sigma}x^{\sigma q}+x^{\sigma q^5})= \alpha^q x^{\sigma q}+\beta^q(\rho^{\sigma q}x^{\sigma q^2}+x^\sigma)+ \]
\[ +\alpha^{q^3} x^{\sigma q^3}+\beta^{q^3}(\rho^{\sigma q^3}x^{\sigma q^4}+x^{\sigma q^2})+b(\alpha^{q^5}x^{\sigma q^5}+\beta^{q^5}(\rho^{\sigma q^5}x^\sigma + x^{\sigma q^4})).\]
This is a polynomial identity in $x^{\sigma}$. Comparing coefficients we get $\alpha=\delta=0$ and
\[ \left\{ \begin{array}{llr} \beta^q \rho^{\sigma q} +\beta^{q^3}=0 ,\\
 \beta^{q^3}\rho^{\sigma q^3}+b\beta^{q^5}=0 .\end{array} \right.\]
Subtracting the second equation from the $q^2$-th power of the first gives $\beta^{q^5}(1-b)=0$, and hence $\beta=0$, a contradiction.
\end{proof}

\begin{theorem}
The maximum scattered $\F_q$-linear set $L_g$ of $\PG(1,q^6)$, where $g$ is defined in Theorem \ref{main}, is not $\PGaL(2,q^6)$-equivalent to any any other previously known maximum scattered $\F_q$-linear set.
\end{theorem}
\begin{proof}
We have to confront $L_g$ with $L^{1,6}_{s}$, $L^{2,6}_{t,\rho}$ and $L^{3,6}_{h,\xi}$.
Suppose that $L_g$ is equivalent to one of these linear sets, then by \cite{LaShZa2013} and by Propositions \ref{classLP6} and \ref{classCMPZ}, respectively,
there exists $\varphi \in \G(2,q^n)$ such that $U_g^{\varphi}$ equals one of $U^{1,6}_{s}$, $U^{2,6}_{t,\rho}$ and $U^{3,6}_{h,\xi}$, a contradiction by Proposition \ref{equiv}.
\end{proof}

For the sake of completeness we show that the $\ZG$-class of $L_g$, defined as in Theorem \ref{main}, is one.

\begin{proposition}\label{classtrin}
The $\ZG$-class of $L_g$ of $\PG(1,q^6)$, where $g(x)=x^q+x^{q^3}+bx^{q^5}$, is at most two if $b\neq 0$.
\end{proposition}
\begin{proof}
Since $g(x)$ and $\hat{g}(x)=b^qx^q+x^{q^3}+x^{q^5}$ define the same linear set, we know $L_g=L_{\hat g}$.
Suppose $L_f = L_g$ for some $f(x)=\sum_{i=0}^5 a_ix^{q^i} \in \F_{q^6}[x]$.
We show that there exists $\lambda\in \F_{q^6}^*$ such that either $\lambda U_f=U_g$ or $\lambda U_f=U_{\hat g}$.

By \eqref{6} we obtain $a_0=0$, by \eqref{7} with $k=1,3$ we have
\[
a_1a_5^q=b^q
\]
and
\begin{equation}
\label{ivt}
a_3^{q^3+1}=1.
\end{equation}
By \eqref{8}, with $k=2$ we have
\[
a_1^{q+1}a_4^{q^2}+a_2a_5^{q+q^2}=0
\]
and taking this into account, together with \eqref{7} applied for $k=2$ we obtain $a_2=a_4=0$.

Using Proposition \ref{Di}, we get that the determinant of the Dickson matrix associated to the $q$-polynomial
\[F(Y)=(a_1x^q+a_3x^{q^3}+a_5x^{q^5})Y-x(Y^q+Y^{q^3}+bY^{q^5})\]
is the zero-polynomial modulo $x^{q^6}-x$. Substituting $a_1=(b/a_5)^q$ it turns out that the coefficient of $x^{1+q+2q^4+2q^5}$ in the reduced form of this determinant is
\[a_3^{q^2}a_5^{-1-q-q^4-q^5}(a_3^{q^3}a_5^{2+q+q^4+2q^5}(a_3^{q+q^4}-1)-\]
\[(a_5^{1+q+q^5}-a_3^q)b^{1+q+q^5}+a_3^{q^4}a_5^{2+2q+2q^5}-a_5^{1+q+q^5}).\]
Applying $a_3^{q^3+1}=1$, it follows that
\[(a_3^q-a_5^{1+q+q^5})b^{1+q+q^5}=a_5^{1+q+q^5}-a_3^{q^4}a_5^{2+2q+2q^5}=
(a_3^q-a_5^{1+q+q^5})a_3^{q^4}a_5^{1+q+q^5}.\]

If $a_3^q=a_5^{1+q+q^5}$, then \eqref{ivt} yields $\N_{q^6/q}(a_5)=1$, and hence there exists $\lambda \in \F_{q^6}^*$ such that $a_5=\lambda^{q^5-1}$. It is easy to see that in this case $\lambda U_f=U_{\hat g}$.

Now suppose $a_3^q\neq a_5^{1+q+q^5}$ and hence $b^{1+q+q^5}=a_3^{q^4}a_5^{1+q+q^5}$. Taking $(q^3+1)$-th powers yields
$\N(b/a_5)=1$ and hence there exists $\lambda \in \F_{q^6}^*$ such that $a_5=b\lambda^{q^5-1}$. It is easy to see that in this case $\lambda U_f=U_g$.

\end{proof}

\begin{corollary}
The $\ZG$-class of $L_g$ of $\PG(1,q^6)$, where $g(x)=x^q+x^{q^3}+bx^{q^5}$, is two if $b^2+b=1$. In particular, it is two if $g$ is defined as in Theorem \ref{main}.
\end{corollary}
\begin{proof}
If $\lambda U_g=U_{\hat g}$ for some $\lambda\in \F_{q^6}$, then $\lambda g(x)={\hat g}(\lambda x)$ for each $x\in \F_{q^6}^*$ and hence comparing coefficients gives $b^q\lambda^{q-1}=1$ and $\lambda^{q^3-1}=1$. Then $b=\lambda^{q^2-1}$ and hence $\N_{q^6/q^2}(b)=1$. Also, $b\in \F_{q^2}$ from which $b^3=1$ follows, contradicting $b^2+b=1$.
\end{proof}

\section{\texorpdfstring{New MRD-codes}{New MRD-codes}}
\label{sec:MRD}

The set of $m \times n$ matrices $\F_q^{m\times n}$ over $\F_q$ is a rank metric $\F_q$-space
with rank metric distance defined by $d(A,B) = rk\,(A-B)$ for $A,B \in \F_q^{m\times n}$.
A subset $\cC \subseteq \F_q^{m\times n}$ is called a \emph{rank distance code} (RD-code for short). The minimum distance of $\cC$ is
\[d(C) = \min_{{A,B \in \cC},\ {A\ne B}} \{ d(A,B) \}.\]

In \cite{Delsarte} the Singleton bound for an $m\times n$ rank metric code $\cC$ with minimum rank distance $d$ was proved:
\[
\#\cC \leq q^{\max \{m,n\}(\min \{m,n\}-d+1)}.
\]
If this bound is achieved, then $\cC$ is an \emph{MRD-code}.

When $\cC$ is an $\F_q$-linear subspace of $\F_q^{m\times n}$, we say that $\cC$ is an \emph{$\F_q$-linear code} and the
dimension $\dim_q (\cC)$ is defined to be the dimension of $\cC$ as a subspace over $\F_q$.
If $d$ is the minimum distance of $\cC$ we say that $\cC$ has parameters $(m,n,q;d)$.


We will use the following equivalence definition for codes of $\F_q^{m \times m}$. If $\cC$ and $\cC'$ are two codes then they are \emph{equivalent} if and only if
there exist two invertible matrices $A,B \in \F_q^{m \times m}$ and a field automorphism $\sigma$ such that
$\{A C^\sigma B \colon C\in \cC\}=\cC'$, or $\{A C^{T\sigma}B \colon C\in \cC\}=\cC'$, where $T$ denotes transposition.
The code $\cC^T$ is also called the \emph{adjoint} of $\cC$.

\medskip

In \cite[Section 5]{Sh} Sheekey showed that scattered $\F_q$-linear sets of $\PG(1,q^n)$ of rank $n$ yield $\F_q$-linear MRD-codes with parameters $(n,n,q;n-1)$. We briefly recall here the construction from $\cite{Sh}$. Let $U_f=\{(x,f(x))\colon x\in \F_{q^n}\}$ for some $q$-polynomial $f(x)$.
Then, after fixing an $\F_q$-bases $\{b_1,\ldots,b_n\}$ for $\F_{q^n}$ we can define an isomorphism between the rings $\mathrm{End}(\F_{q^n},\F_q)$ and $\F_q^{n\times n}$. More precisely, to $f\in \mathrm{End}(\F_{q^n},\F_q)$ we associate the matrix $M_f$ of $\F_q^{n\times n}$ with $i$-th column
$(a_{1,i},\ldots,a_{n,i})^T$, where $f(b_i)=\sum_{j=1}^n a_{j,i} b_j$.\footnote{In the paper \cite{LTZ2} the anti-isomorphism $f \mapsto M_f^{T}$ is considered.}
In this way the set
\[
\cC_f:=\{x\mapsto af(x)+bx \colon a,b \in \F_{q^n}\}
\]
corresponds to a set of $n\times n$ matrices over $\F_q$ forming an $\F_q$-linear MRD-code with parameters $(n,n,q;n-1)$. Also, since $\cC_f$ is an $\F_{q^n}$-subspace of $\mathrm{End}(\F_{q^n},\F_q)$, its \emph{middle nucleus} $\cN(\cC)$ (cf. \cite{LTZ2}, or \cite{LN2016} where the term \emph{left idealiser} was used)
is the set of scalar maps $\cF_n:=\{x\in\F_{q^n}\mapsto \alpha x\in\F_{q^n}\colon \alpha\in\F_{q^n}\}$, i.e. $\cN(\cC_f)\cong \F_{q^n}$.
Note that equivalent codes have isomorphic middle nuclei. For further details see \cite[Section 6]{CMPZ}.

Let $\cC_f$ and $\cC_h$ be two MRD-codes arising from maximum scattered subspaces $U_f$ and $U_h$ of $\F_{q^n}\times \F_{q^n}$.
In \cite[Theorem 8]{Sh} the author showed that there exist invertible matrices $A$, $B$ and $\sigma \in \mathrm{Aut}(\F_{q})$ such that $A \cC_f^\sigma B=\cC_h$ if and only if $U_f$ and $U_h$ are  $\Gamma\mathrm{L}(2,q^n)$-equivalent. 

\begin{theorem}
\label{thm:newMRD}
The $\F_q$-linear MRD-code $\cC_g$ arising from the maximum scattered $\F_q$-subspace $U_g$, $g$ as in Theorem \ref{main}, with parameters $(6,6,q;5)$ and with middle nucleus isomorphic to $\F_{q^6}$ is not equivalent to any previously known MRD-code.
\end{theorem}
\begin{proof}
From \cite[Section 6]{CMPZ}, the previously known $\F_q$-linear MRD-codes with parameters $(6,6,q;5)$ and with middle nucleus isomorphic to $\F_{q^6}$, up to equivalence, arise from one of the following maximum scattered subspaces of $\F_{q^{6}}\times\F_{q^{6}}$:
$U^{1,6}_{s}$, $U^{2,6}_{s,\delta}$, $U^{3,6}_{s,\delta}$. From Proposition \ref{equiv} the result follows.
\end{proof}

\section*{Acknowledgments}
The first author is very grateful for the hospitality of the Department of Mathematics and Physics, University of Campania "Luigi Vanvitelli", Caserta, Italy,  where he was a visiting researcher for 3 months during the development of this research. The third author also thanks for the hospitality of the Institute of Mathematics, E\"otv\"os Lor\'and University, Budapest, Hungary, where he spent 3 months as a PhD student during this work.

\noindent Bence Csajb\'ok\\
MTA--ELTE Geometric and Algebraic Combinatorics Research Group\\
ELTE E\"otv\"os Lor\'and University, Budapest, Hungary\\
Department of Geometry\\
1117 Budapest, P\'azm\'any P.\ stny.\ 1/C, Hungary\\
{{\em csajbokb@cs.elte.hu}}
\bigskip

\noindent Giuseppe Marino, Ferdinando Zullo\\
Dipartimento di Matematica e Fisica,\\
Universit\`a degli Studi della Campania ``Luigi Vanvitelli'',\\
Viale Lincoln 5, I-\,81100 Caserta, Italy\\
{{\em giuseppe.marino@unicampania.it}, {\em ferdinando.zullo@unicampania.it}}
\bigskip

\end{document}